\documentclass[a4paper,11pt,reqno]{amsart}
\usepackage[utf8]{inputenc}
\usepackage{amsmath}
\usepackage{amsthm}
\usepackage{amssymb}
\usepackage{amsfonts,mathrsfs}
\usepackage{stmaryrd}
\usepackage{amsxtra}  
\usepackage{epsfig}
\usepackage{verbatim}
\usepackage{enumerate}
\usepackage{xcolor}
\usepackage{enumitem}
\usepackage{bigints}
\usepackage{mathrsfs}
\usepackage{tikz-cd}
\usepackage{bm,bbm}
\usepackage{hyperref}
\usepackage[all]{xy}
\usepackage{mathtools, hyperref}
\usepackage{tikz}
\usepackage{tikz-cd}
\usetikzlibrary{shapes}
\usetikzlibrary{plotmarks}

\theoremstyle{plain}
\newtheorem{theorem}[equation]{Theorem}
\newtheorem{proposition}[equation]{Proposition}

\theoremstyle{definition}
\newtheorem{definition}[equation]{Definition}

\theoremstyle{remark}

\theoremstyle{remark}

\numberwithin{equation}{section}

\newcommand{\tmop}[1]{\ensuremath{\operatorname{#1}}}
\renewcommand{\Re}{\tmop{Re}}

\newcommand{\abs}[1]{\left\vert#1\right\vert}

\newcommand{\ol}{\overline}

\newcommand{\ipr}[1]{\left\langle #1 \right\rangle}

\newcommand{\xdownarrow}[1]{%
  {\left\downarrow\vbox to #1{}\right.\kern-\nulldelimiterspace}
}

\newcommand{\cx}{{\mathbb C}}

\newcommand{\rl}{\mathbb R}

\newcommand{\miso}{\bm{\mu}}

\newcommand{\B}{{\mathbb B}}

\newcommand{\g}{\mathfrak{g}}

\title{A Hermitian metric on hyperbolic complex manifolds}
\author{Debraj Chakrabarti}
\address{Department of Mathematics,
	Central Michigan University,
	Mt. Pleasant, MI 48859,
	USA}
\email{chakr2d@cmich.edu}
\author{Prachi Mahajan}
\address{Department of Mathematics,
        Indian Institute of Technology  Bombay, 
        Powai,
        Mumbai 400 076,
        India.}   
\email{prachi.mjn@iitb.ac.in}

\thanks{The first-named author was partially supported by a  US National Science
	Foundation grant number DMS-2153907, and by a grant from the Simons 
	Foundation (number 706445).}

\keywords{}
\subjclass[2020]{32F45}

\begin{document}
\begin{abstract} We describe a method of defining a Hermitian metric on Kobayashi hyperbolic manifolds. The metric is distance decreasing
under holomorphic mappings, up to a multiplicative constant. This method is distinct from the classical construction of Wu, and yields a metric which is expected 
to have superior regularity properties.
\end{abstract}
\maketitle

\section{Introduction}
Invariant distances and metrics (i.e., the infinitesimal forms of these distances on tangent vectors) play an important role in understanding the behavior of holomorphic maps of higher dimensional
complex manifolds. The most important examples include the Kobayashi-Royden, Carath\'{e}odory-Reiffen and Bergman metrics and the corresponding distance functions, though there are many more.
While all these are invariant in the sense that 
they are preserved under biholomorphisms, the Kobayashi and Carath\'{e}odory metrics and distances are distinguished by the remarkable property 
that they do not  increase under holomorphic maps. From 
the well-known  geometric interpretation of the Schwarz Lemma due 
to Ahlfors, such a ``distance-decreasing" property is characteristic 
of holomorphic maps of Hermitian manifolds of strongly negative holomorphic curvature, i.e. of Hermitian manifolds in which the holomorphic sectional curvature is bounded above by a negative constant. For general information 
on invariant metrics and their applications see, e.g, \cite{mahabharata,abate}.

The Carath\'{e}odory and Kobayashi metrics are  Finsler metrics which are Hermitian only in very special manifolds.
Kobayashi therefore raised a natural question: whether each Kobayashi hyperbolic manifold (i.e. for which the Kobayashi pseudodistance is a distance)  admits a Hermitian metric of 
strongly negative holomorphic curvature (see \cite{kobayashi1}). Partly  in an attempt to attack 
this problem, Wu (see \cite{wu1993}) introduced 
an invariant Hermitian metric  on Kobayashi hyperbolic manifolds $M$, which is \emph{almost distance-decreasing}, where  for an invariant
metric $M\mapsto F_M$ on a class of complex manifolds, to say
$F$ is almost distance-decreasing means
that for each positive integer $m$ there is a $C_m>0$ such that
$f^* F_N \leq C_m F_M$, whenever $f:M\to N$ is
a holomorphic map and $\dim_\cx(M)=m$ (see \cite{kim,mahabharata} for more information on the Wu metric.). For the Wu metric, $C_m = \sqrt{m}$ and  for the classical distance-decreasing metrics such as the Kobayashi-Royden and Carath\'{e}odory-Reiffen metrics, we have $C_m=1$ for each $m$.

It is conjectured that the Wu metric may provide a positive
answer to Kobayashi's question, i.e. the Wu metric is a Hermitian metric 
of strongly negative holomorphic curvature on Kobayashi hyperbolic manifolds.
Unfortunately, except in very special cases (see \cite{cheungkim1,cheunkim2, cheungkimconstant, balaprachiwu1,balaprachiwu2}),  it has proved difficult to understand the holomorphic curvature of the Wu metric, and Kobayashi's question remains 
unanswered. On the other hand, Wu's idea of replacing the Kobayashi metric with an almost distance-decreasing  Hermitian metric
constructed out of the Kobayashi metric seems like the 
correct approach to Kobayashi's question. In view of this, one may
ask for alternatives to Wu's construction which also lead to 
almost distance-decreasing Hermitian metrics on Kobayashi hyperbolic manifolds.

The aim of this note is to point out that a recent construction 
in Finsler geometry provides an alternative to the Wu metric, with very similar formal properties (but admittedly equally mysterious curvature behavior).
This construction, introduced in \cite{binet}, yields a Hermitian metric 
$\kappa_M$ on a Kobayashi hyperbolic complex manifold $ M $. It has interesting applications to Finsler geometry, and will be discussed in detail in Section~\ref{sec-binet} below. The following properties
of the new metric follow from the construction:

\begin{theorem}
    \label{thm-main}
    On every Kobayashi hyperbolic complex manifold $M$, there is 
    defined an invariant Hermitian metric $\kappa_M$ which has the following properties:
    \begin{enumerate}
     
        \item On the unit ball $\B_n$ of $\cx^n$, the metric $\kappa_{\B_n}$ coincides with ${ \frac{1}{n+1} B}$
        where $B$ is the Poincar\'{e}-Bergman metric.

\item 
If the Kobayashi-Busemann metric of $M$ is of class $\mathcal{C}^k$ for some $k$, then the metric $\kappa_M$ is also of class $\mathcal{C}^k$.
        
        \item If $M$ is Kobayashi complete, then $\kappa_M$ is a complete continuous
        Hermitian metric.
        \item If $f:M\to N$ is a holomorphic map of Kobayashi hyperbolic manifolds, then 
                \[ f^*\kappa_N \leq n^{n+1}m^{m+1} \kappa_M, \]
where $m=\dim_\cx M$ and $n=\dim_\cx N.$ If the Kobayashi metric on $ M$ is already Hermitian, we in fact have
    \[f^*\kappa_N\leq n^{n+1}\kappa_M. \]
 
    \end{enumerate}
\end{theorem}
 In fact, the construction of the metric $\kappa_M$ can be applied to arbitrary complex manifolds (not necessarily Kobayashi hyperbolic manifolds), but then we will obtain a pseudohermitian metric (i.e. not necessarily positive definite). For instance, if the Kobayashi metric vanishes identically on $ M $, then so does $ \kappa_M $. However, where the alternative construction 
 is expected to behave better is with respect to regularity properties, as one can see from the following example. Consider the geometrically convex domains 
\[
E_{2m}= \{ (z_1, z_2) \in \mathbb{C}^2: |z_1|^{2m} + |z_2|^2 < 1 \} 
\]  
for $ m > 1 $. Then the Kobayashi-Busemann metric (which coincides with the Kobayashi metric here) is of class $ C^2 $ on $ E_{2m} \times (\mathbb{C}^2 \setminus \{0\} ) $ (see Theorem 1.1 of \cite{Ma}). It follows from part (2) of  Theorem \ref{thm-main} above that $ \kappa_{E_{2m}} $ is also of class $ C^2 $.
However, the Wu metric on $ E_{2m} $ is $ C^1$-smooth but not $ C^2$-smooth as shown in Theorem 2 of \cite{cheungkim1}. Consequently,
we can talk meaningfully about the curvature and other differential properties of the new metric $\kappa_M$.
This suggests a program of deeper study of $\kappa_M$ in the context of the Kobayashi conjecture, but this will need
understanding of the smoothness of the Kobayashi and Kobayashi-Busemann metrics, which is available for some 
classes of domains. 
 For instance, if $ M $ is a bounded strongly convex domain with $ C^k (k \geq 6) $ boundary, then the Kobayshi-Busemann metric (which is the same as the Kobayashi metric) on $ M $ is $ C^{k-5}$-smooth (\cite{Lempert}) and hence the metric $ \kappa_M $ is also $ C^{k-5}$-smooth.

The Wu metric behaves nicely with respect to formation of products and with respect to covering maps. Similar properties hold for $\kappa_M$:

\begin{theorem}
    \label{thm-product}
    Let $M_1,\dots, M_N$ be Kobayashi hyperbolic manifolds, with 
    $\dim_\cx M_j=n_j.$ Let $M=M_1\times  \dots \times M_N$ be the product manifold and let $n= \sum_{j=1}^N n_j =\dim_\cx M.$
    Then we have a product representation:
    \[ \kappa_M = \frac{n_1+1}{n+1}\kappa_{M_1}\times \dots \times \frac{n_N+1}{n+1}\kappa_{M_{N}}. \]
\end{theorem}

\begin{theorem} \label{thm-covering}
Let $M,N$ be Kobayashi-hyperbolic manifolds and let $\pi:M\to N$
be a holomorphic covering mapping. Then $\pi^*\kappa_N=\kappa_M.$
    
\end{theorem}

\section{From norms to inner products}
\subsection{The real Binet-Legendre construction}
Let $V$ be a vector space over $\rl$ or $\cx$, and let $F$ be a Banach-space norm 
on $V$. 
Recall that the \emph{open unit ball} of the norm $F$ is the set 
$\Omega_F=\{F<1\}\subset V$. The open unit ball is clearly convex and if $\xi \in \Omega_F$, then for each scalar $c$ such that $\abs{c}=1$, we have $c\xi \in \Omega_F$, which is expressed by saying that 
$\Omega_F$ is \emph{symmetric} when the scalars are real, and by saying that it is \emph{balanced} when the scalars are complex.
When the norm $ F $ is derived from a positive definite inner product $\ipr{,}$, i.e. $F(v)= \sqrt{\ipr{v,v}}$, the open unit ball of $F$ is an ellipsoid. 

Consider the (ill-defined) problem of associating with each norm $F$, an inner product that reflects in some way the geometry of $F$, or equivalently, associating an ellipsoid with the  
convex set $\Omega_F$. Many solutions to this problem have been 
suggested, such as using the ellipsoid of the smallest volume containing the set $\Omega_F$, or dually, the ellipsoid of the largest volume contained in $\Omega_F$ (see \cite{binet2}). 
In \cite{binet}, the following construction was used to associate 
an inner product $g_F$ with a norm $F$ on a $d$ dimensional real vector space $V$. Let $g_F^*:V^*\times V^*\to \rl$ be the inner product 
on the dual space $V^*=\hom_{\rl}(V,\rl)$ defined  for $\rl$-linear
functionals $\lambda, \mu : V\to \rl$ by
\[ g_F^*(\lambda, \mu)= \frac{d+2}{\abs{\Omega_F}}\int_{\Omega_F} \lambda(\eta)\mu(\eta)d\eta,\]
where the integral is taken with respect to a Haar measure on $V$ (thought of as a topological group under addition), and $\abs{\Omega_F}$ is the measure of $\Omega_F$ with respect to this measure. Notice that such a Haar measure may be obtained by choosing 
a basis of $V$ to identify it with $\rl^d$ and using the standard Lebesgue measure. Since any two Haar measures differ by a positive multiplicative constant, the average in the above definition is well-defined.

The (real) \emph{Binet-Legendre} inner product $g_F$ is defined to be the dual of $g_F^*$, i.e., for $v,w\in V$
 \[ g_F(v,w)= g_F^*(\check{v}, \check{w}),\]
where $\check{\cdot}$ is the musical isomorphism, i.e. for $v\in V$, $\check{v}$ is the unique element of $V^*$ such that 
\[ g_F^*(\check{v}, \theta)=\theta(v)\]
for each $\theta\in V^*.$  Motivation for considering this inner product arises from various
sources, including mechanics and probability, and it has been used in the solution of important problems in
Finsler geometry (see \cite{binet}).

\subsection{ The complex BL construction}

\label{sec-binet} From this point, all vector spaces are considered to complex. We now extend the construction of the Binet-Legendre inner product to complex vector spaces.  Let $V^*= \hom_\cx(V,\cx)$ be the dual  complex vector space of $V$ and let $n=\dim_\cx V$. We define
for $\theta, \phi\in V^*$:
\begin{equation}
    \label{eq-dualmetric}
    \g_F^*(\theta, \phi)= \frac{n+1}{\abs{\Omega_F}}\int_{\Omega_F} \theta(\eta)\ol{\phi(\eta)}d \eta,
\end{equation}
where, as in the real case, $d\eta$ denotes any Haar measure on $V$, thought of as an additive topological group and $\abs{\Omega_F}$ denotes the measure of the open set $\Omega_F$ in this
Haar measure. 
The function $\g_F^*:V^*\times V^*\to \cx$ in \eqref{eq-dualmetric}  is the \emph{dual complex Binet-Legendre inner product} corresponding to the pseudonorm $F$.

We now define
\[ \g_F:V\times V \to \cx\]
to be the Hermitian inner product dual to $\g^*_F$. We call $\g_F$ the \emph{ complex Binet-Legendre inner product associated with } the norm $F$.

More precisely, let $\miso^{-1}:V^*\to V$ be the 
musical isomorphism induced by $\g_F^*$, i.e. $\phi(\miso^{-1}(\theta))=\g_F^*(\theta,\phi)$ for $\phi,\theta\in V^*$.
Then $\g_F=\miso^* (g_F^*)$, i.e.
\begin{equation}
    \label{eq-BLdef}
    \g_F(v,w)= \g_F^*(\miso(v), \miso(w)).
\end{equation}
The following is easy to show by a direct computation, and establishes a relation between the real and complex 
versions of this construction. 
\begin{proposition} Let $V$ be an $n$-dimensional complex vector space and let $F$ be a norm on $V$. Let $J$ denote the complex
structure of the real vector space $V$,  $J\eta=i\eta$ for $\eta\in V$. Then we have, for $v,w\in V$
\begin{equation}
    \label{eq-realcomplex}
    \g_F(v,w)=g_F(v,w)-i g_F(Jv,w),
\end{equation}
where $\g_F$ and $g_F$ are the complex and real Binet-Legendre inner products associated wth $F$ respectively. In particular, 
\[ \Re \g_F(v,w)=g_F(v,w). \]
\end{proposition}

\subsection{Basic properties}If $H$ is a Hermitian inner product on a complex vector space $V$, we can define a  norm
$F$ on $V$ by setting $F(v)=\sqrt{H(v,v)}.$ We then say that the norm $F$ is \emph{Hermitian} and write $F= \sqrt{H}$.
We have the following basic properties of the complex Binet-Legendre construction above. Most of the proofs are 
analogous to the real case considered  in \cite[Proposition~12.1]{binet}, so we suppress the details. However, 
we do note that the triangle inequality is not important in proving some of the properties of the Binet-Legendre construction. This may be of relevance in applying the construction to invariant Finsler metrics which do not necessarily satisfy the triangle inequality, such as the Kobayashi-Royden pseudometric. 

\begin{proposition} \label{Basic}
Let $F$ be a pseudonorm on an $n$-dimensional
complex vector space $V$, i.e., it satisfies all the properties of a norm, except perhaps the triangle inequality. 
Construct the Hermitian metric $\g_F$ exactly as in Section~\ref{sec-binet} above.
Then:
\begin{enumerate}[wide]
    \item If $ F $ comes from a Hermitian inner product $h$ on $V$, i.e., $ F = \sqrt{h} $, then  $\g_{F}=h.$
\item If $A\in GL(V)$ is a $\cx$-linear automorphism of 
$V$ then $\g_{A^*F}=A^*\g_F$, where $A^* F(v)=F(Av)$ and
$A^*\g_F(u,v)= \g_F(Au, Av).$ 
\item Let $\alpha, \beta >0$ and let $F_1, F_2$ be {pseudo}norms
on $V$ such that 
\begin{equation}
    \label{eq-comparablef1f2} \alpha F_1 \leq F_2 \leq \beta F_1.
\end{equation}
Then for each $v\in V$ we have
\begin{equation} \label{eq-comparablegf}
     \frac{\alpha^{n+1}}{\beta^{n}} \sqrt{\g_{F_1}(v,v)} \leq \sqrt{\g_{F_2}(v,v)}\leq  \frac{\beta^{n+1}}{\alpha^{n}}\sqrt{\g_{F_1}(v,v)}.
\end{equation}
\item One of the two inequalities in \eqref{eq-comparablegf} is a equality for some $v\not=0$
if and only if $F_2=\beta F_1$ for some $\beta>0$. Then we have  $\g_{F_2}= \beta^2 \g_{F_1}.$


\item If $ F $ satisfies the triangle inequality, then
\[\frac{1}{n^{\frac{n+1}{2}}} F \leq \sqrt{\g_F} \leq {n^{\frac{n+1}{2}}}{F}.  \]
\end{enumerate}
\end{proposition}
\begin{proof}
  Parts (1) and (2) are shown exactly as for the real version in \cite[Proposition~12.1]{binet}. 
     \begin{enumerate}[wide]   \setcounter{enumi}{2}
    \item Denote by $\Omega_1, \Omega_2$ the unit balls of the norms $F_1$ and $F_2$ respectively. They  are therefore related by 
    \[\frac{1}{\beta} \Omega_1 \subset \Omega_2 \subset \frac{1}{\alpha}\Omega_1\]
    and therefore
    \[ \frac{1}{\beta^{2n}} \abs{\Omega_1}\leq \abs{\Omega_2}\leq \frac{1}{\alpha^{2n}} \abs{\Omega_1}. \]
    For $\theta\in V^*$ we have
    \begin{equation}
        \label{eq-ineq}
         \frac{1}{n+1}\cdot\g_{F_2}^*(\theta, \theta)=\frac{\int_{\Omega_2}\abs{\theta(\eta)}^2d\eta}{\abs{\Omega_2}}\leq \frac{\int_{\frac{1}{\alpha}\Omega_1}\abs{\theta(\eta)}^2d\eta }{\frac{1}{\beta^{2n}}\abs{\Omega_1}}.
    \end{equation}
Making the linear change of variables $\xi = \alpha \eta$ we therefore get
\[  \frac{1}{n+1}\g_{F_2}^*(\theta, \theta)= \beta^{2n}\frac{\int_{\Omega_1} \abs{\theta( \frac{1}{\alpha}\xi)}^2 \frac{1}{\alpha^{2n}}d\xi}{\abs{\Omega_1}}= \frac{\beta^{2n}}{\alpha^{2n+2}}\frac{\int_{\Omega_1} \abs{\theta( \xi)}^2 d\xi}{\abs{\Omega_1}}=\frac{\beta^{2n}}{\alpha^{2n+2}}\cdot \frac{1}{n+1}\g_{F_1}^*(\theta, \theta). \]
Combining with a similar computation for  $\g_{F_1}^*(\theta, \theta)$ we
obtain as positive definite forms on $V^*$:
\[ \frac{\alpha^{2n}}{\beta^{2n+2}}\g_{F_1}^* \leq \g_{F_2}^* \leq  \frac{\beta^{2n}}{\alpha^{2n+2}}\g_{F_1}^*.\]
Dualizing, we obtain 
\begin{equation*}
    \frac{\alpha^{2n+2}}{\beta^{2n}} \g_{F_1} \leq \g_{F_2}\leq \frac{\beta^{2n+2}}{\alpha^{2n}}\g_{F_1},
\end{equation*}
from which  \eqref{eq-comparablegf} follows immediately.

\item Suppose for definiteness that there is a vector $v\not=0$ in $V$ such that 
$\dfrac{\alpha^{2n+2}}{\beta^{2n}} \g_{F_1}(v,v) = \g_{F_2}(v,v).$ Let $\theta=\miso(v)$,
where $\miso:V\to V^*$ is the musical map in \eqref{eq-BLdef}. Then for this $\theta\not=0$ in $V^*$,
the inequality in \eqref{eq-ineq} must be an equality, which means that ${\int_{\Omega_2}\abs{\theta(\eta)}^2d\eta}= {\int_{\frac{1}{\alpha}\Omega_1}\abs{\theta(\eta)}^2d\eta }$ 
and ${\abs{\Omega_2}}= {\frac{1}{\beta^{2n}}\abs{\Omega_1}}$. Since $\theta$ is nonzero, the set 
on which it vanishes has measure zero and consequently we must have  $\Omega_2=\frac{1}{\alpha}\Omega_1$ and also $\Omega_2=\frac{1}{\beta}\Omega_1,$ so $\alpha=\beta$. Then from  \eqref{eq-comparablef1f2}  we see that $F_2=\beta F_1.$ Finally \eqref{eq-comparablegf} gives
\[ \beta^2 \g_{F_1} \leq \g_{\beta F_2} \leq \beta^2 \g_{F_1}.\]
The case in which there is equality in the other inequality of \eqref{eq-comparablegf} is handled similarly.
\item Let $\Omega$ denote the unit ball of $F$, so that 
$\Omega$ is a convex balanced set in the $n$ complex dimensional vector space $V$.
Recall the theorem of John (\cite{john,ball}) that there is a unique Hermitian ellipsoid $J$ of minimal volume such that 
\[ \Omega \subset J \subset \sqrt{n} \Omega.\]
Denote by $\mathfrak{h}$ the Minkowski function of $J$ which is therefore a Hermitian inner product on $V$. From the above inclusions we see that 
\begin{equation}
    \label{eq-bounds}
    \frac{1}{\sqrt{n}}F \leq \sqrt{\mathfrak{h}} \leq F.
\end{equation}
Applying \eqref{eq-comparablegf} and the fact $\g_{\sqrt{\mathfrak{h}}}={\mathfrak{h}}$ we have,\[
\frac{1}{n^{\frac{n+1}{2}}} \sqrt{\g_F} \leq \sqrt{\mathfrak{h}} \leq {n^{\frac{n}{2}}} \sqrt{\g_F}.\]
Combining this with \eqref{eq-bounds}, the result follows. 
    \end{enumerate}

\end{proof}
The following proposition relates the Binet-Legendre construction with construction of products, and will be used for the proof of Theorem~\ref{thm-product}.

\begin{proposition} \label{prop-product}
    Let a finite dimensional complex vector space $V$ be represented as a direct sum $V=\bigoplus_{k=1}^p V_k$ of vector subspaces $V_k$, and for each $k$ let $F_k$ be a pseudonorm on $V_k$. Define a pseudonorm $F$ on $V$ by 
    \[F=\max_{1\leq k \leq p} \pi_k^* F_k\]
    where $\pi_k: V\to V_k$ is the projection onto the $k$-th factor. 
    Then we have 
    \[ \g_F= \sum_{j=1}^p\frac{n_j+1}{n+1} \pi_j^*\g_{F_j}, \]
    where $n=\dim V, n_k=\dim V_k$.
    
\end{proposition}
\begin{proof} It is easily verified that $F$ is a pseudonorm, and that 
\[ \Omega= \Omega_1\times \dots \times \Omega_p,\]
where $\Omega=\{v\in V: F(v)<1\}$ is the pseudoball of $F$ and $\Omega_k=\{v\in V_k: F_k(v)<1\}$ is the pseudoball of $V_k$. For $\eta\in V,$
write $\eta_k=\pi_k\eta\in V_k$ for simplicity so that $\eta= \sum_{k=1}^p \eta_k$. For a functional $\theta\in V^*$, denote by $\theta_k$ the restriction of $\theta$ to the    subspace $V_k$. Then $\theta= \sum_{k=1}^p \theta_k$. For $\theta, \phi\in V^*$, notice that 
\begin{align*}
    \g_F^*(\theta, \phi)&=\frac{n+1}{\abs{\Omega}} \int_\Omega \theta(\eta)\ol{\phi(\eta)}d\eta\\
    &= \frac{n+1}{\abs{\Omega}} \int_\Omega \theta\left(\sum_{j=1}^p \eta_j\right)\ol{\phi\left(\sum_{k=1}^p \eta_k\right)}d\eta\\
    &= \frac{n+1}{\abs{\Omega}} \sum_{j,k=1}^p \int_\Omega \theta_j(\eta_j)\ol{\phi_k(\eta_k)}d\eta\\ 
    &=\sum_{j=1}^p \mathrm{I}_j + \sum_{j\not=k}\mathrm{II}_{j,k}, 
\end{align*}
where 
\begin{align*}
    \mathrm{I}_j&=\frac{n+1}{\abs{\Omega}}\int_\Omega \theta_j(\eta_j)\ol{\phi_j(\eta_j)}d\eta,\\
    \mathrm{II}_{j,k}&= \frac{n+1}{\abs{\Omega}}\int_\Omega \theta_j(\eta_j)\ol{\phi_k(\eta_k)}d\eta.
\end{align*}
Now we can write the Haar measure $d\eta$ on $V$ as a product measure $d\eta=d\eta_1\dots d\eta_p$, where each $d\eta_k$ is a Haar measure on $(V_k,+).$ Since by symmetry $\int_{\Omega_j}\alpha(v)dv=0$ for each $j$ and
each $\alpha\in V_j^*$, we see that 
\[ \mathrm{II}_{j,k}= \frac{n+1}{\abs{\Omega}}\cdot \left(\int_{\Omega_j}\theta_j(\eta_j)d\eta_j\right)\times \left(\int_{\Omega_k}\ol{\phi_k(\eta_k)}d\eta_k \right)\times \prod_{\ell\not=j,k} \abs{\Omega_\ell}=0.\]
Also,
\begin{align*}
    \mathrm{I}_j&= \frac{n+1}{\abs{\Omega}}\cdot \left(\int_{\Omega_j}\theta_j(\eta_j)\ol{\phi_j(\eta_j) }d\eta_j\right)\times \prod_{\ell\not=j} \abs{\Omega_\ell}\\
    &= \frac{n+1}{n_j+1}\cdot\left(\frac{n_j+1}{\abs{\Omega_j}}\cdot \int_{\Omega_j}\theta_j(\eta_j)\ol{\phi_j(\eta_j) }d\eta_j \right)\\
    &=\frac{n+1}{n_j+1}\cdot \g_{F_j}^*(\theta_j,\phi_j)\\
    \end{align*}
Therefore
\[  \g_F^*(\theta, \phi)= \sum_{j=1}^p \frac{n+1}{n_j+1}\cdot \g_{F_j}^*(\theta_j,\phi_j),\]
so that  we have
\[ \g_F= \sum_{j=1}^p\frac{n_j+1}{n+1} \pi_j^*\g_{F_j}. \] 
\end{proof}   

\section{The metric $\kappa_M$} 
\subsection{The Kobayashi-Royden pseudometric} For an extensive introduction to the invariant metrics of complex analysis, see \cite{mahabharata}. Here we will
recall the definitions and properties that are relevant to the new metric. 

Let $ M $ be a complex manifold of dimension $ n $. Let $ p \in M $ and $ v \in \textbf{T}_p M $ be a holomorphic tangent vector at $ p $. The infinitesimal Kobayashi(-Royden) pseudometric is defined as
\[
k_M(v) = \inf\left\{ \left.\frac{1}{\lambda_f}\right| \lambda_f>0, f:\Delta \to M \text{ holomorphic}, f(0)=p, f'(0)=\lambda_f v \right\},
\]
where $\Delta$ denotes the unit disc. 
Recall that
\begin{itemize}
 \item $ k_M(v) \geq 0 $ for all $ v \in \textbf{T}_p M$,
 \item $ k_M (c v) = |c| k_M(v) $ for all $ c \in \mathbb{C} $,
 \item $ k_M $ is upper semi-continuous on the holomorphic tangent bundle $ \textbf{T} M $.
\end{itemize}
It follows that the pseudoball 
\begin{equation} \label{Kob-indica}
\Omega_{M,p} = \{ v \in \textbf{T}_p M: k_M(v) < 1 \}
\end{equation} 
determined by $ k_M $ (referred to as the \textit{indicatrix} of $ k_M $ at $ p $) is open and balanced. However, in general, the triangle inequality does not hold for $ k_M (\cdot) $ on $ \textbf{T}_p M $.

Further recall that $ M $ is said to be Kobayashi hyperbolic if $ k_M $ is locally bounded below by a positive constant (\cite{Roy}), i.e., for each $ p \in M $, there is a neighbourhood $ U $ of $ p $ and a positive constant $ C $ such that
\begin{equation} \label{Kob-hypo}
k_M (v) \geq C \|v\| 
\end{equation}
for all $ p \in U $  and $ \| \cdot\| $ a Hermitian norm on $ \textbf{T}_p M $. In particular, if $ M $ is Kobayashi hyperbolic, then 
\begin{itemize}
   \item  $ k_M(v) =0 $ if and only if $ v = 0 $.
   \item  The pseudoball determined by $ k_M $ is bounded. Indeed, let $  C > 0 $ be as asserted in \eqref{Kob-hypo}. It follows that the pseudoball $ \{ k_M(v) < 1 \} $ is contained in the Hermitian ball $ \{ \|v\| < 1/C \} $, and consequently it is bounded.
\end{itemize}
The integrated Kobayashi pseudodistance $ d_M $ on $ M $ is defined as follows: for $ p_1, p_2 \in M $,
\[
d_M(p_1,p_2) = \inf_{\gamma} \int k_M \left( \dot{\gamma}(t) \right)dt
\]
where the infimum is taken over all piecewise differentiable curves $ \gamma $ in $ M $ joining $ p_1 $ to $ p_2 $. Then $ d_M $ is called the integrated form of $ k_M $. Finally, $ M $ is said to be Kobayashi complete hyperbolic if $ M $ is hyperbolic and $ (M, d_M) $ is Cauchy complete. 

\subsection{The Kobayashi-Busemann metric and the definition of $\kappa_M$}
The Kobayashi-Royden pseudometric displays some undesirable features, 
such as the fact that it does not satisfy the triangle
inequality on the tangent spaces, though the integrated pseudodistance $d_M$ does satisfy the triangle inequality,
thanks to the infimum over curves in its definition. This is because the definition of $k_M$ corresponds to an infinitesimal form of the so-called \emph{Lempert function},
\[\ell_M(z,w)=\inf\left\{ \left.d_\Delta(0, f(\zeta))\right|\zeta\in \Delta,  f:\Delta \to M \text{ holomorphic}, f(0)=z, f(\zeta)=w \right\},\]
wgere $d_\Delta$ denotes the Poincaré hyperbolic distance on the unit disc. The Lempert function does not satisfy the triangle inequality, and 
to obtain a genuine triangle-inequality-satisfying metric from it, one can either use a well-known construction involving a chain of discs, or use
the definition in terms of the infinitesimal metric given in the previous section. Similarly,
the infinitesimal form of the Lempert function, the Kobayashi-Royden pseudometric, does
not satisfy the triangle inequality, but the triangle inequality can be salvaged by a further 
modification using classical ideas of Busemann and Meyer (see \cite{busemann}), thus obtaining 
the Kobayashi-Busemann infinitesimal
pseudometric (see \cite{kobayashi3}).

Let us now recall the definition of the infinitesimal Kobayashi-Busemann pseudometric $ \hat{k}_M $ on a complex manifold $ M $ at a point $ p \in M $:
\[
\hat{k}_M(v) = \inf\{ t > 0, t^{-1} v \in \hat{\Omega}_{M, p}\},
\]
for $ v \in \textbf{T}_p M $, where $ \hat{\Omega}_{M, p} $ denotes the convex hull of the pseudoball $ {\Omega}_{M, p}$ with respect to the Kobayashi pseudometric $ k_M $ at $ p $. It follows, by construction, that $\hat{\Omega}_{M, p} $ is the unit ball determined by $ \hat{k}_M $ at $ p $. Note that
\begin{itemize}
 \item $ \hat{k}_M(v) \geq 0 $ for all $ v \in \textbf{T}_p M$,
 \item $ \hat{k}_M (c v) = |c| k_M(v) $ for all $ c \in \mathbb{C} $,
 \item $ \hat{k}_M (v) \leq k_M (v) $ for all $ v \in \textbf{T}_p M$,
 \item $ \hat{k}_M (v + v') \leq \hat{k}_M(v) + \hat{k}_M(v')$ for all $ v, v' \in \textbf{T}_p M$.
 \end{itemize}
In particular, $ \hat{k}_M $ is continuous on $ \textbf{T}_p M $. Consequently, the associated unit ball $ \hat{\Omega}_{M, p} $ is an open, bounded, balanced convex subset of $ \textbf{T}_p M $. This then leads to the main definition of this paper:
\begin{definition} 
Let $ M $ be a Kobayashi hyperbolic complex manifold. The metric $\kappa_M$ is defined by application of the complex Binet-Legendre construction pointwise to the Finsler metric $ \hat{k}_M$, i.e., for 
     $p\in M$ and $v,w\in \textbf{T}_p M$, set
    \[\kappa_M(v,w)=\g_{ \hat{k}_M}(v,w). \]
\end{definition}

\subsection{Comparison with the Wu metric}

We begin by recalling the definition of the Wu metric and its relation with the Kobayashi-Busemann metric. Let $ M $ be a Kobayashi hyperbolic complex manifold of complex dmension $m$. The Wu metric $ w_M$ at a point $ p \in M $ is the Hermitian inner product on $ \textbf{T}_pM $ whose unit ball is the complex ellipsoid with minimum volume among all ellipsoids containing the Kobayashi indicatrix $ \Omega_{M,p} $. Like other Hermitian metrics, $w_M$ defines a Finsler metric by 
taking the norm, which we denote by $\sqrt{w_M}$, i.e. $\sqrt{w_M}(v)=\sqrt{w_M(v,v)}$ for each $v\in \textbf{T}_pM$.

Now, the complex ellipsoid with minimum volume that contains $ \Omega_{M,p} $ is the same as the ellipsoid with minimum volume containing its convex hull $ \hat{\Omega}_{M,p} $. This implies that $ \hat{k}_M \leq \sqrt{w_M}$. Further, $ \hat{\Omega}_{M,p} $  is a convex balanced subset of $ \textbf{T}_p M $ so the theorem of John (\cite{john,ball}) applies, and hence the complex ellipsoid of minimum volume which contains $ \hat{\Omega}_{M,p} $ is contained in $ \sqrt{m} \; \hat{\Omega}_{M,p} $. In other words, the unit ball of the Wu metric $ w_M $ at $ p $ is contained in $ \sqrt{m} \;\hat{\Omega}_{M,p} $ and therefore $\sqrt{w_M} \leq \sqrt{m} \; \hat{k}_M $. It follows that 
\begin{equation} \label{Wu}
\hat{k}_M \leq \sqrt{w_M} \leq \sqrt{m} \; \hat{k}_M.
\end{equation}
On the other hand, it is immediate from Proposition \ref{Basic} (5) that 
\[
\frac{1}{m^{\frac{m+1}{2}}} \hat{k}_M \leq \sqrt{{\kappa}_M} \leq {m^{\frac{m+1}{2}}}{\hat{k}_M}. 
\]
Combining the above two observations, we arrive at the following:
\[
\frac{1}{m^{\frac{m+2}{2}}} \sqrt{w_M} \leq \sqrt{{\kappa}_M} \leq {m^{\frac{m+1}{2}}}\sqrt{{w_M}}, 
\]
which shows in particular that the integrated distance induced by the Wu metric is comparable to that induced by the metric $\kappa_M$. 

\section{Proofs of the Theorems}
\subsection{Proof of Theorem \ref{thm-main}} 
If $ f: M \rightarrow N $ is a biholomorphic mapping of Kobayashi hyperbolic manifolds $ M, N $, then 
\[
f^* \kappa_N = \kappa_M.
\]
Indeed, since $ f $ is a biholomorphism, for each $ p \in M $, the mapping $ f'(p): \textbf{T}_p M \rightarrow \textbf{T}_{f(p)} N $ is a $ \mathbb{C} $-linear isomorphism. In this setting, Proposition \ref{Basic}\;(2) guarantees that $ \mathfrak{g}_{f^* \hat{k}_N} = f^*\mathfrak{g}_{ \hat{k}_N} $. Moreover, $ f $ preserves the Kobayashi-Busemann metric, i.e. $ f^* \hat{k}_N = \hat{k}_M $, which, in turn, implies that the associated complex Binet-Legendre metrices coincide:
\[
\mathfrak{g}_{f^* \hat{k}_N} = \mathfrak{g}_{ \hat{k}_M}.
\]
Combining the above two observations immediately yields that $ f^*\kappa_N = \kappa_M $, i.e., the associated complex Binet-Legendre metric is invariant under biholomorphisms.

Next, part (1) follows from the explicit expression for the Kobayashi metric on $ \mathbb{B}^n $ (which coincides with the Kobayashi-Busemann pseudometric $ \hat{k}_{\mathbb{B}_n} $) and Proposition \ref{Basic}\;(1). 

\begin{enumerate} [wide] \setcounter{enumi}{1}
\item 


If the Kobayashi-Busemann pseudometric $ \hat{k}_M $ is $ C^k$-smooth, then $ C^k$-smoothness of $ \kappa_M $ follows exactly as for the real version in \cite[Theorem~2.4]{binet}.

\item First, note that if $ M $ is complete Kobayashi hyperbolic, then the Kobayashi metric $ k_M $ is continuous on the holomorphic tangent bundle $ \textbf{T} M $. Consequently, Kobayashi-Busemann metric $ \hat{k}_M $ is also continuous on $ \textbf{T} M $ (cf. Proposition 3.5.38, \cite{kobayashifat}). It follows as in part (2) that the associated Binet-Legendre metric $ \kappa_M $ is also continuous. 

To establish the completeness of $ \kappa_M $, notice that 
Proposition \ref{Basic} (5) ensures that 
\[
\frac{1}{m^{\frac{m+1}{2}}} \hat{k}_M \leq \sqrt{\g_{\hat{k}_M}} \leq {m^{\frac{m+1}{2}}}{\hat{k}_M},  
\]
where $ m = \mbox{dim}_{\mathbb{C}} M $. It follows that the integrated form of $ \g_{\hat{k}_M}$ is equivalent to the integrated $ \hat{k}_M $ distance. Since the integrated form of $ \hat{k}_M $ is equal to the Kobayashi pseudodistance $ d_M $ (see \cite{Kobayashi2}), the desired result follows.

\item 
By the distance decreasing property of the Kobayashi-Busemann metric under holomorphic maps, it follows that $ f^* \hat{k}_N \leq \hat{k}_M $. 
It follows from Proposition \ref{Basic} (5) that 
\[
\frac{1}{m^{\frac{m+1}{2}}} \hat{k}_M \leq \sqrt{\g_{\hat{k}_M}} \leq {m^{\frac{m+1}{2}}}{\hat{k}_M},  
\]
and 
\[
\frac{1}{n^{\frac{n+1}{2}}} \hat{k}_N \leq \sqrt{\g_{\hat{k}_N}} \leq {n^{\frac{n+1}{2}}}{\hat{k}_N},  
\]
where $ m = \mbox{dim}_{\mathbb{C}} M $ and $ n = \mbox{dim}_{\mathbb{C}} N $. For $ p \in M $ and $ v \in \textbf{T}_p M $,
observe that 
\begin{align*}
    f^* \g_{\hat{k}_N}  (v,v) 
    =  \g_{\hat{k}_N}  \left( f'(p)v, f'(p) v \right) 
    & \leq  \left( n^{\frac{n+1}{2}} {\hat{k}_N} \left( f'(p)v \right) \right)^2 \\
    & \leq \left( n^{\frac{n+1}{2}} {\hat{k}_M} ( v) \right)^2 \\
    & \leq  n^{n+1} m^{m+1} \g_{\hat{k}_M}( v).      
\end{align*}
This is exactly the assertion $ f^* \kappa_N \leq  n^{n+1} m^{m+1} \kappa_M $.

Lastly, if the Kobayashi metric on $ M$ is already Hermitian, then so is the Kobayashi-Busemann metric on $ M$. This follows from the fact that the Kobayashi-Busemann metric $ \hat{k}_M $ is the double dual of the Kobayashi metric $ k_M $. Applying Proposition \ref{Basic} (1), we get that $ \g_{\hat{k}_M} = ( \hat{k}_M )^2 $. As a consequence,
\[
 \g_{\hat{k}_N}  \left( f'(p)v, f'(p) v \right) \leq  \left( n^{\frac{n+1}{2}} {\hat{k}_N} \left( f'(p)v \right) \right)^2 \leq \left( n^{\frac{n+1}{2}} {\hat{k}_M} ( v) \right)^2  = n^{n+1} \mathfrak{g}_{\hat{k}_M}(v),
\]
or equivalently that
\[
f^* \g_{\hat{k}_N}  (v,v) \leq n^{n+1} \mathfrak{g}_{\hat{k}_M}(v).
\]
 \end{enumerate}

\subsection{Proof of Theorem \ref{thm-product}}
    
Theorem \ref{thm-product} is an immediate consequence of Proposition \ref{prop-product} since the Kobayashi-Busemann metric at any point $ p = (p_1, \dots, p_N) $ of the product manifold $M=M_1\times  \dots \times M_N$ satisfies the following formula:
\[
\hat{k}_{M} (v_1, \dots, v_N) = 
\max \{ \hat{k}_{M_1} (v_1), \dots, \hat{k}_{M_N} (v_N)\},
\]
where $ v_i \in \mathbf{T}_{p_i}M_i $ for each $ i = 1, \dots, N $ (see Proposition 3.5.25 of \cite{kobayashifat}).

\subsection{Proof of Theorem \ref{thm-covering}}

We first recall the property of the Kobayashi-Busemann metric under covering projections (cf. Proposition 3.5.26, \cite{kobayashifat}): Let $ N $ be a complex manifold and $ M $ be a covering manifold of $ N $ with covering projection $ \pi: M \rightarrow N $. Then 
\[
\hat{k}_M = \pi^* \hat{k}_N.
\]
The goal is to show that a similar relation holds for the associated Binet-Legendre metrics under the additional hypothesis that $ M, N  $ are Kobayashi hyperbolic. To this end, first note that, since $ \pi $ is a local biholomorphism, for each $ p \in M $, the mapping $ \pi'(p): \textbf{T}_p M \rightarrow \textbf{T}_{ \pi(p)} N $ is a $ \mathbb{C} $-linear isomorphism. Hence Proposition \ref{Basic}\;(2) applies so that 
\[
\mathfrak{g}_{\pi^* \hat{k}_N} = \pi^*\mathfrak{g}_{ \hat{k}_N}, 
\]
which, in turn, implies that
\[
\mathfrak{g}_{\hat{k}_M} = \pi^*\mathfrak{g}_{ \hat{k}_N}, 
\]
i.e., $ \kappa_M = \pi^* {\kappa_N} $.
\section{Some questions}
At this point we do not know if the metric $\kappa_M$ provides
better insight to Kobayashi's problem than the Wu metric. At present, 
there are very few concrete situations where one can actually compute the Wu metric. The most important example is that of the Thullen domains (\cite{cheungkim1,cheunkim2,cheungkimconstant, balaprachiwu1,balaprachiwu2}). It would be extremely interesting to see what happens to the metric $\kappa_M$ on these domains, and how it compares to the Wu metric. 

\subsection*{Acknowledgements:} We gratefully acknowledge the comments of Vladimir Matveev. The first-named author thanks
the Indian Institute of Technology, Bombay for its hospitality during 
his stay there as a visiting faculty in 2024.
\bibliographystyle{alpha}
\bibliography{kobayashi}
\end{document}